\documentclass{ip-journal}
\usepackage{amsmath,amssymb,graphicx,latexsym,mathrsfs,hyperref,verbatim}
\usepackage[all]{xy}

\newtheorem{thm}{Theorem}[section]
\newtheorem{cor}[thm]{Corollary}
\newtheorem{lem}[thm]{Lemma}

\theoremstyle{definition}
\newtheorem{defn}[thm]{Definition}

\theoremstyle{remark}
\newtheorem{rem}[thm]{Remark}

\numberwithin{equation}{section}

\def\Cb{{\mathbb C}}\def\Gb{{\mathbb G}}\def\Rb{{\mathbb R}}\def\Zb{{\mathbb Z}}
\def\af{{\mathfrak a}}\def\bfr{{\mathfrak b}}\def\lf{{\mathfrak l}}\def\nf{{\mathfrak n}}\def\of{{\mathfrak o}}\def\sf{{\mathfrak s}}
\def\tf{{\mathfrak t}}
\def\gf{{\mathfrak g}}\def\nf{{\mathfrak n}}
\def\Pc{{\mathcal P}}\def\Rc{{\mathcal R}}

\newcommand{\tof}[1]{\stackrel{#1}{\rightarrow}}

\newcommand{\Tof}[1]{\stackrel{#1}{\longrightarrow}}
\newcommand{\lTof}[1]{\stackrel{#1}{\longleftarrow}}

\newcommand{\LG}{G^\vee}

\renewcommand{\b}[1]{\textbf{#1}}

\newcommand{\wt}{\widetilde}
\newcommand{\T}{\mathtt{T}}
\DeclareMathOperator{\Span}{Span}
\DeclareMathOperator{\Hom}{Hom}
\DeclareMathOperator{\End}{End}

\DeclareMathOperator{\ch}{ch}
\begin{document}
\title{T-duality for Langlands dual groups}
\author{Calder Daenzer}
\author{Erik Van Erp}
\email{calder.daenzer@gmail.com}
\email{jhamvanerp@gmail.com}
\maketitle
\begin{abstract} This article addresses the question of whether Langlands duality for complex reductive Lie groups may be implemented by $\T$-dualization.  We prove that for reductive groups whose simple factors are of Dynkin type A, D, or E, the answer is yes.
\end{abstract}

\section{Introduction}
\noindent This article addresses the following question:
\vspace{.1in}
\begin{center}\emph{\ \quad Is the relationship between a complex reductive group and its Langlands dual\\ a form of T-duality?}
\end{center}
\vspace{.1in}

This question arises because a reductive group can be interpreted as a principal torus bundle for a maximal torus, and the maximal tori of Langlands dual groups are canonically $\T$-dual.

Let us make the above question more precise.
Suppose $G$ is a complex reductive group.  Choosing a maximal torus $T\subset G$ allows one to define a root datum of $G$:
\[ \Rc(G,T):=\{ C\subset \Lambda,\ R\subset \Lambda^\perp \}. \]
Here $R$ and $C$ are the sets of roots and coroots of $G$, while $\Lambda^\perp:=\Hom(T,\Cb^\times)$ and $\Lambda:=\Hom(\Cb^\times,T)$ are its character and co-character lattices.

A root datum determines the reductive group $G$ up to isomorphism.
Interchanging the roles of roots and coroots and of the character and co-character lattices results in a new root datum:
\[ \Rc(G,T)^\vee:= \{ R\subset \Lambda^\perp,\ C\subset \Lambda \}.\]
The Langlands dual group of $G$ is the complex reductive group $G^\vee$ (unique up to isomorphism) determined by the dual root datum $\Rc(G,T)^\vee$.

A root datum also implies a choice of maximal torus $T\subset G$ via the canonical isomorphism
\[ T\simeq \Hom(\Lambda^\perp,\Cb^\times), \]
and likewise a natural choice of maximal torus for the Langlands dual group $G^\vee$:
 \[ T^\vee:=\Hom(\Lambda,\Cb^\times)\subset G^\vee. \]
By definition, the  torus $T^\vee$ is the $\T$-dual  of $T$.
\footnote{This is the noncompact $\T$-dual torus $T^\vee\simeq (\Cb^\times)^n$.  To obtain a compact version, which is more common in the $\T$-duality literature, one simply begins with a compact form of $G$ instead of a complex form.}

The translation action of a maximal torus makes $G\to G/T$ a principal torus bundle.  Thus given that the maximal tori of Langlands dual groups are $\T$-dual, it is natural to ask whether $G$ viewed as a principal $T$-bundle, and $G^\vee$ viewed as a principal $T^\vee$-bundle, are $\T$-dual to one another.

A problem that immediately presents itself is that $G\to G/T$ and $G^\vee\to G^\vee/T^\vee$ are torus bundles over different spaces, whereas $\T$-duality is a relationship between torus bundles over a common base space.  However, if the simple factors of $G$ are all of Dynkin type $A$, $D$, or $E$, then the manifolds $G/T$ and $G^\vee/T^\vee$ are isomorphic.
Therefore, for such groups the question of $\T$-duality for Langlands duals is well-defined.

But even if the base manifolds $G/T$ and $G^\vee/T^\vee$ are isomorphic, $G$ and $G^\vee$ only have a hope of being $\T$-dual if we twist the duality by introducing $NS$-fluxes on each space. An $NS$-flux on a principal $T$-bundle is a closed, $T$-invariant, integral 3-form on the total space of the bundle. Indeed, a pair of torus bundles without $NS$-flux are $\T$-dual if and only if both torus bundles are flat, which is generally not the case for the bundles $G\to G/T$ and $G^\vee\to G^\vee/T^\vee$.

The main theorem of this paper asserts (roughly) that when the flag manifolds $G/T$ and $G^\vee/T^\vee$ are isomorphic  there is a natural choice of $NS$-flux on $G$ and on $G^\vee$ such that the two groups are $\T$-dual.  To make this more precise, we should say what the relevant $NS$-fluxes are:

Recall that on every Lie group $G$ there is a natural 3-form called the Cartan 3-form.  It is the left invariant 3-form defined by
\[ H(X,Y,Z):=K(X,[Y,Z]),\qquad X,Y,Z\in \gf,\]
where $K$ is the Killing form of $G$.  This form is indeed integral and closed, so it is an $NS$-flux on $G$.

\begin{thm}[Main Theorem]  Suppose that $G$ is a reductive Lie group whose simple Lie algebra factors are all of type ADE.  Then the principal bundles with $NS$-flux $(G\to G/T,H)$ and $(G^\vee\to G^\vee/T^\vee, H^\vee)$ are $\T$-dual.  Here $H$ and $H^\vee$ denote the Cartan 3-forms of $G$ and $G^\vee$.
\end{thm}

\vskip 6pt
\noindent {\bf Interpretations.}
Philosophically, there are two possible interpretations of our main theorem:
\begin{enumerate}
\item Langlands duality is an example of $\T$-duality for principal torus bundles.
\item $\T$-duality for a single torus is an example of Langlands duality.
\end{enumerate}
The first is of course the direct interpretation.  To arrive at the second viewpoint, observe that the Cartan 3-form vanishes on a torus, so that the theorem applied to $G=T$ reduces to the duality between a torus and its $\T$-dual.

Viewing Langlands duality as a generalization of $\T$-duality for a single torus leads to a question for future study: just as the duality relation on tori exhibits new phenomena when studied in families (for instance the appearance of gerbes and noncommutative duals), Langlands duality should exhibit new features in families as well.  In other words, it would be natural to study duality for fibrations with fiber isomorphic to $(G\to G/T,H)$, and the effect of Langlands dualization on such objects.

\vskip 6pt
\noindent {\bf $\T$-duality and Langlands duality in the literature.}
In this article we work in the context of differential geometric T-duality.  This duality induces isomorphisms of (T-invariant) twisted deRham complexes and twisted generalized geometrical structures on the T-dual spaces \cite{GC}, but does not provide a K-theoretic isomorphism.  After our first preprint appeared, Bunke and Nikolaus \cite{BN} were led to study the topological analogue of this problem, which provides the isomorphism of K-theory.

$\T$-duality and Langlands dual groups have also appeared side by side in the semi-classical limit of the geometric Langlands conjecture (\cite{DP},\cite{HT}).  However, that is a very different use of $\T$-duality than what appears in this work: in the semi-classical Langlands context, $\T$-duality relates the Hitchin fibrations of $G$-Higgs and $G^\vee$-Higgs bundles over a common Hitchin base, whereas for us, $\T$-duality directly implements the relation between a group and its Langlands dual.  Thus, while we do expect that there is a connection between these two applications of $\T$-duality, it is certainly not a direct one.

\vskip 6pt
\noindent\b{Outline.} In Sections \eqref{S:Langlands} and \eqref{S:T-duality} we recall the relevant definitions and set up the notation surrounding Langlands duality and T-duality.  In Section \eqref{S:MainTheorem} we prove the main results.
\vskip 6pt
\noindent\b{Acknowledgements.} The authors would like to thank Nigel Higson, Jonathan Block, Ping Xu, and Mathai Varghese for helpful discussions and insights.  We would also like to thank Ulrich Bunke for pointing out an error in an earlier draft of this paper.

\section{Langlands duality}\label{S:Langlands}
Langlands duality for complex reductive groups is a classical notion, but its central constructions, especially for the case of non-semisimple groups, are somewhat scattered in the literature.  Thus for the convenience of the reader we will collect in this section the basic definitions and constructions that are relevant for our purposes.  Further details may be found in \cite{Che}, \cite{SGA}, and \cite{Ste}.
\subsection{Reductive Lie algebras}
Let $\gf$ be a finite dimensional Lie algebra over the field $\Cb$ of complex numbers.
The Lie algebra $\gf$ is {\em reductive} if to each ideal $\af$ in $\gf$ corresponds an ideal $\bfr$ in $\gf$ with $\gf=\af\oplus\bfr$.
The radical of $\gf$ is the maximal solvable ideal in $\gf$, denoted ${\rm rad}\,\gf$.
If $\gf$ is reductive, then ${\rm rad}\,\gf$ is the center of $\gf$, and
\[ \gf = {\rm rad}\gf \oplus [\gf, \gf]\]
while the derived Lie algebra $[\gf, \gf]$ is semisimple.
Any Lie algebra of complex matrices   that is closed under the operation conjugate transpose is reductive, which includes the classical Lie algebras $\gf\lf(n,\Cb)$, $\sf\lf(n,\Cb)$, $\sf\of(n,\Cb)$, etc.

Let $\tf$ denote a maximal abelian subalgebra of a reductive Lie algebra $\gf$.
(The choice of $\tf$ is unique up to an inner automorphism of $\gf$.)
Note that $\tf\supset {\rm rad}\gf$.
The adjoint representation ${\rm ad}\colon \gf\to {\rm End}\,\gf$ restricted to $\tf$ is diagonalizable,
and we have a decomposition
\[ \gf = \tf \oplus \bigoplus_{\alpha\in R} \gf_\alpha\]
Here  $R\subset \tf^*$ denotes the set of roots of $\gf$, while $\gf_\alpha$ is the eigenspace for the root $\alpha\in R$, that is,
\[ X\in \gf_\alpha \iff [h,X] = \alpha(h)X,\quad \forall h\in \tf.\]
The Killing form $K(-,-)$ is the complex bilinear form on $\gf$ defined by
\[ K(X,Y) = {\rm Tr}\,({\rm ad}X\, {\rm ad}Y).\]
The Killing form is non-degenerate when restricted to the semisimple part $[\gf, \gf]$.
This allows us to define a {\em coroot} $h_\alpha\in \tf$ for every root $\alpha \in \tf^*$ as the unique element in $[\gf,\gf]\cap \tf$ satisfying
\[ \alpha=2\frac{K(-,h_\alpha)}{K(h_\alpha,h_\alpha)}. \]
We denote the set of coroots by $C\subset \tf$.
Note that the coroots span $\tf$ as a complex vector space if and only if $\gf=[\gf,\gf]$, that is, if $\gf$ is semisimple.

\subsection{Reductive Lie groups}
Complex semisimple Lie algebras are classified by root systems (equivalently, by Dynkin diagrams, or by Cartan matrices).
This immediately implies a classification of reductive Lie algebras because of the decomposition $\gf={\rm rad}\, \gf\oplus [\gf, \gf]$.
Chevalley showed that semisimple {\em groups} are classified if we consider the root system as a subset of the {\em character lattice} of a maximal torus  \cite{Che}.
Demazure introduced the notion of a {\em root datum}, comprising roots and character lattice as well as coroots and cocharacter lattice, as an elegant tool for classifying {\em reductive} groups \cite{SGA}.
We briefly review these notions here.\footnote{The classification of reductive groups by root data is typically developed for {\em algebraic} groups.
However, over the field of complex numbers, every reductive Lie group is reductive algebraic. In fact, the category of complex reductive Lie groups is equivalent to the category of complex reductive algebraic groups (see e.g. \cite{Lee}).}

The radical of a Lie group
is its maximal connected normal solvable subgroup.
A complex Lie group is {\em semisimple} if it is connected and its radical is trivial,
and {\em reductive} if it is connected and its radical is a torus.\footnote{The precise definition of ``reductive Lie group'' varies among authors, depending on the context and purpose of the definition.  Knapp, for example, allows for finitely many connected components \cite{Kna}, and adds a number of extra requirements.  We follow Steinberg \cite{Ste}, whose definition is suitable for the purpose of classification by root data.}
In this context, the term ``torus'' refers to a complex Lie group that is isomorphic to the direct product of copies of the multiplicative group $\Cb^\times=\Cb\setminus \{0\}$.

Let $G$ be a complex reductive Lie group with Lie algebra $\gf$.
If $T\subset G$ is a maximal torus in $G$, then the Lie algebra $\tf$ of $T$ is maximal abelian in $\gf$,
and we have corresponding sets of roots $R\subset \tf^*$ and coroots $C\subset \tf$ relative to  $T$.
The {\em character lattice} is the free abelian group
\[ \Lambda^\perp = \Hom(T, \Cb^{\times})\]
where $\Hom$ refers to morphisms of complex Lie groups.
The standard identification of the Lie algebra of $\Cb^{\times}$ with $\Cb$ via the  exponential ${\rm exp}\colon \Cb\to \Cb^\times, z\mapsto e^z$ gives an inclusion
\[ \Lambda^\perp=\Hom(T, \Cb^\times)\subset \Hom(\tf, \Cb) = \tf^*.\]
Note that, by definition, one has $R\subset \Lambda^\perp$.
Likewise, the {\em cocharacter lattice} of $T\subset G$ is the free abelian group
\[ \Lambda = \Hom(\Cb^\times,T) \subset \Hom(\Cb, \tf) = \tf,\]
and it can be shown that the coroots satisfy $C\subset \Lambda$.

A canonical pairing between characters and cocharacters
\[ \Lambda \times \Lambda^\perp \to \Zb\]
is defined in the obvious way by composition
\[ \Hom(\Cb^\times, T) \times \Hom(T, \Cb^\times) \to \Hom(\Cb^\times, \Cb^\times) = \Zb,\]
and this agrees with the pairing of $\tf$ and its dual vector space $\tf^*$.
The pairing between $\Lambda$ and $\Lambda^\perp$ is perfect, making $\Lambda^\perp$ the dual lattice of $\Lambda$.  Note that the kernel of the exponential map $\tf\to T$ is the lattice $2\pi i \Lambda$.

The {\em root datum} of the connected complex Lie group $G$ with maximal torus $T$ is the data
\[ \Rc(G,T):=\{ R\subset\Lambda^\perp, C\subset\Lambda\},\]
which implicitly also includes the pairing $\Lambda\times \Lambda^\perp\to \Zb$
and the bijection $R\to C, \alpha\mapsto h_\alpha$ between roots and coroots.
Root data classify complex reductive Lie groups, in the sense that two such groups are isomorphic if and only if their root data are isomorphic (in the obvious sense)
\cite{SGA}, \cite{Ste}.

\subsection{Abstract root data}
For every abstract root datum there exists a unique complex reductive Lie group whose root datum $\{ R\subset\Lambda^\perp, C\subset\Lambda\}$ is isomorphic to the given abstract root datum.
For this statement to make sense, we must specify axioms for an abstract root datum.
The definition is due to Demazure \cite{SGA}.

\begin{defn}\label{defn_root_datum}
An \b{abstract root datum} consists of a lattice $\Lambda$ of finite rank, its dual lattice $\Lambda^*$, and finite subsets $X\subset \Lambda$, $X^*\subset \Lambda^*$ in bijective correspondence $x\mapsto x^*$, such that
\begin{enumerate}
\item $\langle x, x^*\rangle = 2$ for all $x\in X$,
\item for all $y\in X$ the map $x\mapsto x-\langle x, y^*\rangle x$ permutes the elements of $X$; likewise, for all $y^*\in X^*$ the map $x^*\mapsto x^*-\langle y, x^*\rangle x^*$ permutes the elements of $X^*$,
\item if $x\in X$ and $cx\in X$ then $c=\pm 1$.
\end{enumerate}
\end{defn}
If $G$ is a complex reductive group, then the root datum constructed in the previous section satisfies the axioms of Definition \ref{defn_root_datum}.
Conversely, for every abstract root datum \linebreak $\{ X^*\subset \Lambda^*, X\subset \Lambda\}$ there exists a unique-up-to-isomorphism complex reductive group whose root datum is isomorphic to the given one.

To prove existence of a reductive group for given abstract root datum, first of all one checks that the conditions of an abstract root datum ensure that $X^*$ is, in fact, a root system (\cite{SGA} Exp.\ XXI).
The construction of the semisimple Lie algebra that corresponds to the root system $X^*$ is described in most introductory texts on semisimple Lie algebras (for example \cite{Kna}, \cite{Ser}).
The construction of a reductive Lie algebra from a root datum is a minor modification of this construction, but usually not discussed in the standard texts. We therefore briefly review this construction, referring for details to \cite{SGA}.

Given an abstract root datum $\{X^*\subset \Lambda^*,X\subset\Lambda\}$, the maximal abelian subalgebra $\tf$ of the reductive Lie algebra $\gf$ that we wish to construct is the complex vector space
\[ \tf := \Lambda \otimes_\Zb \Cb\]
The duality between $\Lambda$ and $\Lambda^*$ allows us to identity $\tf^* =  \Lambda^* \otimes_\Zb \Cb$ with the vector space dual of $\tf$.
The radical ${\rm rad}\gf$ of $\gf$ is the subspace of $\tf$ that is perpendicular to the set of roots,
\[ {\rm rad}\,\gf =\{h\in \tf \mid \langle h, x^* \rangle = 0, \forall x \in X^*\}.\]
If $\tf_{ss}=\Span_\Cb\{X\}$ denotes the subspace of $\tf$ spanned by the coroots, then
\[ \tf = {\rm rad}\,\gf \oplus \tf_{ss}.\]
Let $\gf_{ss}=\tf_{ss}\oplus \bigoplus_{\alpha\in X^*}\gf_\alpha$ be a {\em semisimple} Lie algebra with maximal abelian subalgebra $\tf_{ss}$, coroots $X\subset \tf_{ss}$ and roots $X^*$ (the latter interpreted as a subset of $\tf_{ss}^*$).
Then the desired reductive Lie algebra is the direct sum
\[ \gf := {\rm rad}\,\gf \oplus \gf_{ss}\]
Lie's third theorem provides a connected, simply-connected complex Lie group $\wt{G}$ (unique up to isomorphism)
corresponding to  $\gf$.
Let
\[L:=\exp(2\pi i\Lambda)\]
denote the discrete subgroup of $\wt{G}$ obtained by exponentiating the cocharacter lattice.
Then the quotient $G:=\wt{G}/L$
is the desired reductive Lie group with root datum $\{X^*\subset\Lambda^*,X\subset \Lambda\}$.

In summary, there is a one-to-one correspondence between isomorphism classes of (connected) complex reductive Lie groups and isomorphism classes of abstract root data.

\subsection{Langlands duality}
It is clear from the symmetric phrasing of Definition \ref{defn_root_datum} that if\linebreak $\{X^*\subset\Lambda^*, X\subset \Lambda\}$ is an abstract root datum, so is its dual $\{X\subset\Lambda, X^*\subset \Lambda^*\}$.
Two connected reductive complex Lie groups $G, G^\vee$ are \b{ Langlands duals} if the root datum of $G^\vee$ is isomorphic to the dual of the root datum of $G$.

\section{$\T$-duality}\label{S:T-duality}
We intend to exhibit a $\T$-duality between Langlands dual groups that are viewed as principal $(\Cb^\times)^n$-bundles via the translation actions of their maximal tori.  However, the standard differential geometric formulation of $\T$-duality is for $U(1)^n$-bundles.  To skirt this apparent incompatibility, we will, when convenient, view Langlands duality as a relation between compact real forms of reductive groups.  This does not lose any information since one recovers the original form by complexification, and it places us in the standard $\T$-duality setup since a principal $(\Cb^\times)^n$-bundle structure on the complex group determines a $U(1)^n$-bundle structure on the compact real form.

Thus in this section we will recall the standard definitions and notation for differential geometric $\T$-duality, phrased for compact tori.

Let $T=V/\Lambda\simeq \Rb^n/\Zb^n$ be a torus, expressed as a vector space $V$ modulo a full rank lattice $\Lambda$.  By definition, the $\T$-dual torus is
\[ T^\vee:=\Hom(\Lambda, U(1))\simeq V^*/\Lambda^\perp.\]
Here $\Lambda^\perp\subset V^*$ denotes the lattice of linear functionals on $V$ which send $\Lambda$ to the integers.\footnote{We choose to view $\Hom(\Lambda,U(1))$ as the definition of the dual.  It is isomorphic to $V^*/\Lambda^\perp\to\Hom(\Lambda,U(1))$ via the map $v^*\mapsto e^{2\pi i v^*(-)}$.  However, for algebraic tori $\Gb_m^n$, the two duals are non-isomorphic, and we would want the dual to be the multiplicative torus.  Thus in the algebraic case we would set $\Lambda:=\Hom(\Gb_m,T)$, and $T^\vee=\Hom(\Lambda,\Gb_m)$.}

There is a line bundle with connection $(\Pc,\nabla_\Pc)$ on $T\times T^\vee$ (the Poincar\'e line bundle), which is used to transport geometric data from $T$ to $T^\vee$.  The total space of $\Pc$ may be defined as the quotient:
\begin{equation}\label{PoncareLineBundle} \Pc:= (V\times T^\vee\times \Cb)/(v,\phi,z)\sim (v+\lambda,\phi,\phi(\lambda)z)\qquad \lambda\in \Lambda. \end{equation}

The connection $\nabla_\Pc$ is defined by the condition that its pullback to the trivial bundle \linebreak $V\times T^\vee\times \Cb$ equals $d+\omega$. Here $\omega$ denotes the tautological 1-form (also called the Liouville 1-form) that one has on $V\times T^\vee$ after identifying it with the cotangent bundle of $T^\vee$.

In applications to derived categories, the Poincar\'e bundle plays a primary role.  However, in the differential geometric setup one is content to forget about the bundle and only remember its curvature form $F_\Pc$ which is ($i/2\pi$-times) the standard integral symplectic form on $T\times T^\vee$.  This curvature defines a Chern-Weil character $\ch(\nabla_\Pc)\in\Omega^{2*}(T\times T^\vee)$, which induces an isomorphism of de Rham cohomology and also of the translation invariant deRham complexes, via the formula: 
\begin{align} H^\bullet_{DR}(T)\simeq \Omega^\bullet(T)^T\Tof{\sim} \Omega^\bullet(T^\vee)^{T^\vee}\simeq H_{DR}^\bullet(T^\vee) \\
\label{E:t-duality} \Omega^\bullet(T)^T\ni \omega\longmapsto \int_T\omega\wedge \ch(\nabla_\Pc)\in\Omega^\bullet(T^\vee)^{T^\vee}. 
\end{align}

The Poincar\'e line bundle also induces isomorphisms between several other sorts of data, such as (invariant) generalized complex structures, $K$-theory, and (in the complex case) derived categories.

Now let us proceed to the definition of $\T$-duality for principal torus bundles.  It turns out that this generalization is not perfectly straightforward: global non-triviality of a principal bundle results in a defect on the dual side, which is accounted for using a 3-form on the dual bundle known as an $NS$-flux (the necessity of this 3-from can be deduced from several perspectives, we refer the reader for example to \cite{BS},\cite{BHM},\cite{MR} or \cite{Dae} for details).  Despite this twist, it is still possible to keep the perspective that duality is implemented by a 2-form $F$ which plays the role of the curvature of a Poincar\'e bundle.

Here are the definitions:
\begin{defn} An \textbf{NS-flux} (also sometimes called an H-flux) on a principal torus bundle $P$ is a closed, integral, $T$-invariant 3-form $H\in \Omega^3(P)$.\end{defn}
Let $M$ be the base of a principal $T$-bundle $P$, and let $P^\vee$ be a principal $T^\vee$-bundle on $M$, and define the projections:
\[ P\lTof{\pi}P\times_M P^\vee\Tof{\pi^\vee} P^\vee. \]

\begin{defn}\label{D:T-duality} A \textbf{$\T$-duality} between a principal $T$-bundle with $NS$-flux $(P,H\in\Omega^3(P))$ and a principal $T^\vee$-bundle with $NS$-flux $(P^\vee,H^\vee\in\Omega^3(P^\vee))$ is the data of a $T\times T^\vee$-invariant 2-form $F\in \Omega^2(P\times_M P^\vee)^{T\times T^\vee}$ satisfying
\begin{enumerate}
\item $dF=\pi^*H-(\pi^\vee)^*H^\vee$.
\item For each $m\in M$, the restriction $F_m$ to the fiber $P_m\times P^\vee_m$ has integral periods.
\item Let $\tf_M$ and $\tf^\vee_M$ denote the vertical tangent bundles of $P$ and $P^\vee$, then $F$ restricts to a nondegenerate pairing:
\[ F:\tf_M\otimes \tf_M^\vee\to\Rb. \]
\end{enumerate}
We will refer to $F$ as a \textbf{dualizing 2-form}.
\end{defn}
Note that $F_m$ is automatically closed (this follows because $F_m$ is translation invariant and $T\times T^\vee$ is abelian), and thus conditions (2) and (3) imply that $F_m$ is an integral symplectic form on $P_m\times P_m^\vee$.  In particular, $F_m$ is the curvature of some connection on a line bundle over $P_m\times P^\vee_m$.
On the other hand, $F$ itself is not a closed form, so it is not the curvature of any line bundle on $P\times_M P^\vee$.  Rather, in this global setting one has a family over $M$ of line bundles, which fits together to a morphism of gerbes with connection, whose 2-curving is $F$.

Using $e^{\frac{i}{2\pi}F}$ as a proxy for the Chern character of a connection in Equation \eqref{E:t-duality}, one produces an isomorphism of $\Zb/2\Zb$-graded chain complexes (\cite{BHM},\cite{GC})
\[ (\Omega^\bullet(P)^T,d_{DR}+H\wedge\cdot)\Tof{\sim} (\Omega^\bullet(P^\vee)^{T^\vee},d_{DR}+H^\vee\wedge\cdot) \]
which induces an isomorphism of twisted de Rham cohomology $H^\bullet(P,d+H)\simeq H^\bullet(P^\vee,d+H)$.  A T-duality also implies a bijection between the translation invariant twisted generalized geometrical structures on the two sides (\cite{GC} Theorem 3.1).  There are associated maps in twisted K-theory, and twisted derived categories in the complex case, though these maps are only isomorphisms when $F$ satisfies an extra unimodularity condition.

\section{T-duality for Langlands dual groups}\label{S:MainTheorem}
Now that the tools are in place, we will describe how Langlands duality may be regarded as an instance of $\T$-duality.  We should remark that the T-duality between $SU(2)$ and its Langlands dual $SO(3)$ was first calculated in \cite{BEM}, where it appeared as an example of $\T$-duality between Lens spaces. \newline

\noindent\textbf{Notation:} In this section $G$ will denote the compact form of a reductive Lie group $G$, and $T\simeq U(1)^n$ a maximal torus in $G$. The complex groups will be denoted $G_\Cb$ and $T_\Cb$.
\begin{lem} Let $G$ be a compact Lie group containing no simple factors of type $B_n$ or $C_n$. Then a Langlands dual $G^\vee$ also has this property, and the flag manifolds $G/T$ and $G^\vee/T^\vee$ are isomorphic.
\end{lem}
\begin{proof}
The effect of Langlands dualization on a Cartan matrix, whose entries are the numbers $\alpha(h_\beta)$ indexed over a set of simple roots, is simply to transpose the matrix.  For simple Lie algebras of type ADE the Cartan matrix is symmetric, and for $G_2$ and $F_4$ the transpose corresponds to a different ordering of simple roots.  However, transposition of Cartan matrices interchanges the types $B_n$ and $C_n$.  In other words, the effect of dualization on simple Lie alegebras is:
\[ A_n\leftrightarrow A_n,\ B_n\leftrightarrow C_n,\ D_n\leftrightarrow D_n,\ E_n\leftrightarrow E_n,\ G_2\leftrightarrow G_2,\ F_4\leftrightarrow F_4.\]
Only for type $B_n$ and $C_n$ does dualization change the isomorphism type of a simple Lie algebra, thus for reductive Lie algebras containing no factors of type $B_n$ or $C_n$, Langlands dualization results in an isomorphic Lie algebra.

Now suppose $\gf$ has no $B_n$ or $C_n$ factors and choose an isomorphism $\gf\tof{\phi}\gf^\vee$.  This determines an isomorphism of the flag manifolds $G/T\simeq G^\vee/T^\vee$ since the flag manifold is isomorphic to the moduli space of Borel subalgebras of $\gf$.  Alternatively, one obtains an isomorphism $G/T\simeq G^\vee/T^\vee$ by exponentiating $\phi$ to an isomorphism between the universal covers $\wt{G}\tof{\sim}\wt{G^\vee}$, which induces an isomorphism of flag manifolds since $\wt{G}/\wt{T}\simeq G/T$, where $\wt{T}$ denotes the inverse image of $T$ under the covering map $\wt{G}\to G$.
\end{proof}
\begin{rem} The flag manifolds of groups of type $B_n$ and $B_n^\vee=C_n$ cannot be isomorphic.  In fact their higher homotopy groups are non-isomorphic, as can easily be seen by comparing the long exact sequences of homotopy groups associated to $T\to G\to G/T$ and $T^\vee\to G^\vee\to G^\vee/T^\vee$.
\end{rem}

Next, recall that on every Lie group $G$, one has a natural $3$-form called the \b{Cartan 3-form}.
It is the left invariant $3$-form $H\in \Omega^3(G)^G$ defined by
\[ H(X,Y,Z):=-\frac{1}{4\pi^2} K(X,[Y,Z]),\qquad X,Y,Z\in \gf,\]
where $K$ is the Killing form of $G$.

Note that if $H$ is the Cartan form of a compact group $G$, then extending $H$ to $\gf\otimes\Cb$ by complex linearity produces the Cartan form of $G_\Cb$.

\begin{lem} Let $G$ be a compact Lie group, then the Cartan 3-form $H$ is an $NS$-flux on the the principal $T$-bundle $G\to G/T$.
\end{lem}
\begin{proof} The 3-form is well known to be closed and integral, and it is $G$-invariant (thus $T$-invariant) by construction.
\end{proof}

\subsection{A tautological 2-form.} Here we describe a 2-form which is defined on the product of any reductive Lie group with its Langlands dual.  It will be used to prove the $T$-duality relation on Langlands dual groups.

The decomposition $\gf=\tf\oplus \nf$ (where $\nf=\bigoplus_{\alpha\in R} \gf_\alpha$) is orthogonal with respect to the Killing form $K$.
This allows us to extend any root $\alpha\in R\subset \tf^*$ to a linear form (still denoted $\alpha$) on all of $\gf$.
The extended root may be written as:
\[ \alpha=2\frac{K(-,h_\alpha)}{K(h_\alpha,h_\alpha)}\in \gf^* \]
This extension allow roots $\alpha\in \gf^*\simeq \Omega^1(G)^G$ to be viewed as left translation-invariant 1-forms on $G$.
Similarly, for the Langlands dual $G^\vee$, the roots $\alpha^\vee\in (\gf^\vee)^*$ determine $G^\vee$-invarinant 1-forms on $\alpha^\vee\in\Omega^1(G^\vee)^{G^\vee}$.

Note that the the bijection $\alpha\leftrightarrow h_\alpha$ between roots and coroots, when composed with the Langlands identification $C=R^\vee$, gives a preferred bijection between the roots of $G$ and the roots of $G^\vee$:
\[ R\tof{\sim}C= R^\vee,\qquad \alpha\to h_\alpha=\alpha^\vee. \]
For us, $\alpha^\vee$ will always denote the image of $\alpha$ under this bijection.  Similarly, one has a preferred bijection $C\to R= C^\vee$ which we denote by $h_\alpha\mapsto h_\alpha^\vee$.  It is easy to verify that $h_\alpha^\vee=h_{\alpha^\vee}$.

\begin{defn}\label{D:Cartan2} Let $G$ be a complex reductive Lie group with Langlands dual $G^\vee$.
We define the \b{tautological 2-form} on $G\times G^\vee$ as:
\[ F:=-\frac{1}{4\pi^2} \sum_{\alpha\in R}\alpha\wedge \alpha^\vee\in \Omega^2(G\times G^\vee)^{G\times G^\vee}. \]
\end{defn}

\subsection{$\T$-duality between Langlands dual groups} We are finally ready to prove the $T$-duality statement.
We denote the Cartan 3-form on $G$ by $H$ and the Cartan 3-from on the Langland's dual $\LG$ by $H^\vee$.  We say a reductive group is of \b{ADE-type} if its Lie algebra is the sum of an abelian Lie algebra and simple Lie algebras of type A, D, or E.

\begin{lem}\label{L:GoodIsom} Let $G$ be a group of $ADE$-type.  Then there exists an isomorphism of Lie algebras $\gf\tof{\phi}\gf^\vee$ which sends each coroot $h_\alpha\in \tf\subset\gf$ to the coroot $h_\alpha^\vee\in\tf^*\subset\gf^\vee$.
\end{lem}
\begin{proof} The Lie algebra $\gf$ decomposes as a sum of simple Lie algebras and an abelian Lie algebra, and Langlands duality respects the decomposition so the isomorphism can be defined piecewise.  For the abelian part any vector space isomorphism will do.  On the semisimple part the map $h_\alpha\mapsto h_\alpha^\vee$ is an isomorphism of root systems precisely because of the relation $\alpha(h_\beta)=\beta(h_\alpha)$.  An isomorphism of root systems determines an isomorphism of Lie algebras, so we are done.
\end{proof}

\begin{rem} Any two such isomorphisms would differ by an automorphism of $\gf$ which fixes each coroot.  If $\gf$ is semisimple, such an automorphism must be inner and furthermore must be implemented by an element of the torus.  At the level of the flag manifold $G/T$, this automorphism corresponds to left multiplication by an element of $T$.  Thus while there is not a canonical isomorphism $G/T\simeq G^\vee/T^\vee$, the difference between any two is no more than the difference between starting with the group $G$ and starting with $tGt^{-1}$ for some $t\in T$.
\end{rem}
\begin{rem} For $G_2$ and $F_4$, there are of course isomorphisms between the Lie algebra and the Langlands dual Lie algebra, but they cannot be chosen of the form specified in Lemma \eqref{L:GoodIsom}.  Because of this, the proof of the following theorem cannot be applied to $G_2$ and $F_4$. \end{rem}
\begin{thm}\label{T:Main}
Let $G$ be a compact Lie group of ADE-type.  Then the principal bundles with $NS$-flux $(G\to G/T,H)$ and $(G^\vee\to G^\vee/T^\vee, H^\vee)$ are $\T$-dual to one another.  Here $T$ is any maximal torus in $G$, and $H$, $H^\vee$ denote the Cartan 3-forms.
\end{thm}
\begin{proof}
Let us begin by labeling some maps:
\begin{equation}
\xymatrix{
 & G\times G^\vee\ar@/_/[ldd]_q  \ar@/^/[rdd]^{q^\vee} & \\
& G\times_M G^\vee\ar@{^{(}->}[u]^{i} \ar[ld]^\pi \ar[rd]_{\pi^\vee} & \\
G \ar[rd]_p&  & G^\vee\ar[ld]^{p^\vee} \\
  & M &
}
\end{equation}
Here $M=G^\vee/T^\vee$, and $p$ is the composition $G\to G/T\tof{\phi}M$ for some $\phi$ as in Lemma \eqref{L:GoodIsom}.  All of the other maps in the diagram are the obvious ones.

Let us start with the case that $G$ is semisimple.  We will show that the pullback $i^*F$ implements $\T$-duality between $(G,H)$ and $(G^\vee, H^\vee)$, where $F$ is the tautological 2-form of Definition \eqref{D:Cartan2}.

Thus the goal is to show that $i^*F$ is a dualizing 2-form in the sense of Definition \eqref{D:T-duality}.
First note that $i^*F$ is indeed $T\times T^\vee$-invariant because it is the restriction of the $G\times G^\vee$-invariant form $F$.  Next we verify conditions (2: Integrality) and (3: Nondegeneracy) of Definition \eqref{D:T-duality}.

To see that $F$ is integral, recall that $2\pi i\Lambda\subset \tf$ is the kernel of the exponential maps $\tf\to T$, and elements in $2\pi i \Lambda$ can be identified with loops in $T$.
Therefore the 1-form $\frac{1}{2\pi i}\alpha$ is integral, and so is
\[ F= \sum_{\alpha\in R}\frac{\alpha}{2\pi i}\wedge \frac{\alpha^\vee}{2\pi i} \]
To check nondegeneracy, it suffices (by translation invariance) to check it for the restriction $F_0$ at the special fiber $T\times T^\vee\subset G\times_M G^\vee$.  To exhibit the nondegeneracy of $F_0:=\sum \alpha\wedge \alpha^\vee\in (\tf\otimes\tf^*)^*\subset\Omega^2(T\times T^\vee)^{T\times T^\vee}$, we will view it as an element of $\End(\tf)$ and show that $F_0(h_\beta)$ is a nonzero multiple of $h_\beta$ for each coroot $h_\beta\in C$.  Thus let $h_\beta$ be such a coroot, and denote by $W_\beta\subset W$ the stabilizer of $h_\beta$ in the Weyl group.

Now on the one hand, the only vectors that are stabilized by all of $W_\beta$ are the linear multiples of $h_\beta$, but for any $w\in W_\beta$, we have:
\begin{align*}
 w(F_0(h_\beta))&=\sum_{\alpha\in R} \alpha(h_\beta)w(h_\alpha) \\
                &=\sum_{\alpha\in R}\alpha(w^{-1}(h_\beta))w(h_\alpha)&(\text{using } h_\beta=w^{-1}(h_\beta))\\
                &=\sum_{w(\alpha)\in R} w(\alpha)(h_\beta)h_{w(\alpha)}=F_0(h_\beta).
\end{align*}
so $F_0(h_\beta)=m_\beta h_\beta$ for some number $m_\beta$.  It must be shown that $m_\beta\neq 0$.  To that end, we recall the classic formula for the angle between roots:
\[ \alpha(h_\beta)h_\alpha(\beta)=4\cos^2{\phi} \]
which is of course non-negative, and we calculate:
\begin{align*}
 m_\beta=\frac{1}{2}m_\beta h_\beta(\beta)&=\frac{1}{2}F_0(h_\beta,\beta)& \\
    &=-\frac{1}{8\pi}\sum \alpha(h_\beta)h_\alpha(\beta)< 0
\end{align*}
In fact, using $\alpha(h_\beta)=\beta(h_\alpha)$, it can be quickly shown that $m_\beta=-\frac{1}{8\pi}K(h_\beta,h_\beta)$.

Now to complete the proof (for $G$ semisimple), it only remains to show that
\[  d(i^*F)=\pi^*H-(\pi^\vee)^* H^\vee\ (\ =i^*q^*H-i^*(q^\vee)^* H^\vee\ ). \]
Because $F$, $q^*H$ and $(q^\vee)^*H^\vee$ are translation invariant, it suffices to verify that
\[ dF=q^*H- (q^\vee)^*H^\vee \]
when restricted to the subspace $E_0:=T_{G\times_M G^\vee,(1,1)}\subset T_{G\times G^\vee,(1,1)}= \gf\oplus\gf^\vee$.  (It is not true that $dF$ equals the difference of $H$ and $H^\vee$ on all of $\gf\oplus\gf^\vee$, so the restriction to $E_0$ is indeed necessary.)

Recall that for a left-invariant 1-form $\omega\in \Omega^1(G)^G$, one has $d\omega(X,Y)=-\omega([X,Y])$.  Thus
\[ -4\pi^2\,dF=\sum d\alpha\wedge \alpha^\vee - \sum\alpha\wedge d\alpha^\vee \]
takes a fairly simple form.
Using this, we will calculate the equality directly using a spanning set of $E_0$.  In fact, it is more convenient to work with a spanning set of the complexification of $E_0$; so we will show that the complex linear extensions of $dF$ and $\pi^*H-(\pi^\vee)^* H^\vee$ are equal (working this way also makes the complexified version of this theorem more obvious).

For each $\xi\in R$, we choose $\ X_\xi\in \gf_\xi$, $Y_\xi\in\gf_{-\xi}$ satisfying $[X_\xi,Y_\xi]=h_\xi$.  Then $E_0\otimes\Cb$ is spanned by the set
\[ S:=\{h_\xi,X_\xi+\phi_*X_\xi,Y_\xi+\phi_*Y_\xi,h_\xi^\vee\}_{\xi\in R}, \]
where $\phi_*$ is the Lie algebra map inducing the isomorphism $G/T\simeq G^\vee/T^\vee$.

Because $q_*h^\vee_\alpha=0$ and because the decomposition $\gf=\tf\oplus \nf$ is orthogonal with respect to the Killing form, the only triples of vectors in $S$ not annihilated by $q^*H$ are (up to reordering) those of the form $(h_\xi,X_\zeta+\phi_*X_\zeta,Y_\zeta+\phi_*Y_\zeta)$.  And for these we have
\[ -4\pi^2\,q^*H(h_\xi,X_\zeta+\phi_*X_\zeta,Y_\zeta+\phi_*Y_\zeta)=K(h_\xi,[X_\zeta,Y_\zeta])
    =K(h_\xi,h_\zeta)=\sum_{\alpha\in R}\alpha(h_\xi)\alpha(h_\zeta). \]
Similarly, the only nonzero values of $(q^\vee)^*H^\vee$ are
\begin{align}\label{E:qveeH}
\ \hspace{-1.5in} -4\pi^2\,(q^\vee)^*H^\vee(h^\vee_\xi,X_\zeta+\phi_*X_\zeta,Y_\zeta+\phi_*Y_\zeta)     &=K^\vee(h^\vee_\xi,[\phi_*X_\zeta,\phi_*Y_\zeta]^\vee) \notag\\
       &=K^\vee(h^\vee_\xi,h^\vee_\zeta)\notag\\
       &=\sum_{\alpha^\vee\in R^\vee}\alpha^\vee(h^\vee_\xi)\alpha^\vee(h^\vee_\zeta)\notag\\
       &=\sum_{\alpha\in R}\xi(h_\alpha)\zeta(h_\alpha).
\end{align}

Next, note that the 1-form $q^*\alpha\in\Omega^1(G\times G^\vee)$ vanishes on all vectors in $S$ except possibly the coroots $\{h_\beta\}$, and $q^*\alpha$ satisfies $dq^*\alpha=q^*d\alpha=q^*\alpha([-,-])$.  And $(q^\vee)^*\alpha^\vee$ has analogous properties, so we have
\[ -4\pi^2\,dF= \sum_{\alpha\in R}(q^*\alpha)([-,-])\wedge (q^\vee)^*\alpha^\vee(-)-\sum_{\alpha\in R}(q^*\alpha)(-)\wedge (q^\vee)^*\alpha^\vee([-,-]) \]
The only nonzero values of the first term in $dF$ are
\begin{align}\label{E:dFfirst term}
\ \hspace{-1in} \sum_{\alpha\in R}(q^*\alpha)([X_\xi+\phi_*X_\xi,Y_\xi+\phi_*Y_\xi])\wedge (q^\vee)^*\alpha^\vee(h^\vee_\zeta)
    &=\sum_{\alpha\in R}\alpha(h_\xi)\alpha^\vee(h^\vee_\zeta)\notag \\
    &=\sum_{\alpha\in R}\alpha(h_\xi)\zeta(h_\alpha).
\end{align}
Finally, since $G$ has only $ADE$-type factors, we have $\alpha(h_\xi)=\xi(h_\alpha)$, so that \eqref{E:qveeH} and \eqref{E:dFfirst term} are equal.

That is, the first term of $dF$ equals $(q^\vee)^*H^\vee$.  A similar calculation shows that the second term of $dF$ equals $-q^*H$, so the theorem is proved for $G$ semisimple.

Now assume that $G$ is not semisimple, and decompose its Lie algebra into the abelian and  semisimple summands: $\gf=Z(\gf)\oplus [\gf, \gf]$.
Then the pair $(G,H)$ and $(G^\vee,H^\vee)$ are still $T$-dual with the same Cartan 3-forms, but to the dualizing 2-form $i^*F$, we must add an new term $F_\Pc$.
It is simply the curvature of the Poincar\'e line bundle on the torus $\exp(\af)\subset G$ and its dual $\exp(\af^\vee)\subset G^\vee$.
The point is that $i^*F$ fails the nondegeneracy condition when $G$ is not semisimple, and adding $F_{\Pc}$ simply restores nondegeneracy.  But because $dF_{\Pc}=0$, there is no need to alter the $NS$-fluxes.
\end{proof}
\begin{cor}\label{C:Main} Theorem \eqref{T:Main} is valid with $G$ replaced by its complexification, that is, a complex reductive group of ADE-type, together with its Cartan 3-form as NS-flux, is $\T$-dual to its Langlands dual.
\end{cor}
\begin{proof} This is because the dualizing form of Theorem \eqref{T:Main} is also a dualizing form for the complexification of a compact Lie group.
\end{proof}
\begin{cor} Theorem \eqref{T:Main} is also true using any nonzero multiple of the Cartan forms, that is, for $0\neq n\in\Zb$, the pair $(G,nH)$ and $(G^\vee, nH)$ are $\T$-dual\end{cor}
\begin{proof} One simply uses $n$-times the dualizing 2-form of Theorem \eqref{T:Main}.\end{proof}

\end{document}